\documentclass[reqno]{amsart}

\usepackage[utf8]{inputenc} 
\usepackage{amssymb}
\usepackage{graphicx}
\usepackage{latexsym}
\usepackage{stmaryrd}
\usepackage{enumitem}
\usepackage[bookmarks,hyperfootnotes=false]{hyperref}
\usepackage{multirow}
\usepackage{xcolor}
\usepackage{caption}
\colorlet{darkishRed}{red!80!black}
\colorlet{darkishBlue}{blue!60!black}
\colorlet{darkishGreen}{green!60!black}
\usepackage{hyperref}
\hypersetup{
pdftitle = {Duality theorems for stars and combs},
pdfsubject = {Duality theorems for stars and combs},
pdfauthor = {Carl Bürger and Jan Kurkofka},
pdfkeywords = {},
draft = false,
bookmarksopen=true}

\usepackage{theoremref} 
\usepackage{tikz}
\usepackage{tikz-cd}
\usetikzlibrary{calc,through,intersections,arrows, trees, positioning, decorations.pathmorphing, cd}
\usepackage{relsize}
\usepackage{comment}
\usepackage{xifthen}
\usepackage{mathrsfs}
\usepackage{svg}
\usepackage{mathtools}
\usepackage{nccmath}
\usepackage{pifont}

\newcommand{\SCnumberCite}[1]
{%
\ifthenelse{\equal{#1}{1}}{\cite{StarComb1StarsAndCombs}}{}%
\ifthenelse{\equal{#1}{2}}{\cite{StarComb2TheDominatedComb}}{}%
\ifthenelse{\equal{#1}{3}}{\cite{StarComb3TheUndominatedComb}}{}%
\ifthenelse{\equal{#1}{4}}{\cite{StarComb4TheUndominatingStar}}{}%
}

\newcommand{\SCnumberHand}[2]
{%
\ifthenelse{\equal{#1}{#2}}{\ding{43}\,}{}%
}

\newcommand{\SCintroList}[2]
{%
    \ifthenelse{\equal{#1}{#2}}{(this paper)}{\SCnumberCite{#1}}%
}

\newcommand{\SCintroDetermined}[1]
{%
    \ifthenelse{\equal{#1}{1}}{In this paper, we determine}{In the first paper of this series, we determined}%
}

\newcounter{quotecount}

\makeatletter
\newcommand*{\addFileDependency}[1]{
  \typeout{(#1)}
  \@addtofilelist{#1}
  \IfFileExists{#1}{}{\typeout{No file #1.}}
}
\makeatother

\newcommand{\Abs}[1]{\partial_{\Omega} {#1}}

\newcommand{\rest}{\upharpoonright}

\renewcommand{\subset}{\subseteq}
\renewcommand{\supset}{\supseteq}

\def\TFAD{Let $G$ be any connected graph and let $U\subset V(G)$ be any vertex set. Then the following assertions are complementary:}

\newcommand{\at}{attached to }

\newcommand{ \N } { \mathbb{N} }

\newcommand{\dblue}[1]{\textcolor{darkishBlue}{#1}}

\newcommand{\dc}[1]{\lceil #1\rceil}
\newcommand{\uc}[1]{\lfloor #1\rfloor}
\newcommand{\guc}[1]{\lfloor\mkern-1.4\thinmuskip\lfloor #1\rfloor\mkern-1.4\thinmuskip\rfloor}
\newcommand{\nt}{T_{\textup{\textsc{nt}}}}

\newcommand{\Top}{\textsc{Top}}
\newcommand{\MTop}{\textsc{MTop}}
\newcommand{\VTop}{\textsc{VTop}}

\makeatletter

\def\calCommandfactory#1{%
   \expandafter\def\csname c#1\endcsname{\mathcal{#1}}}
\def\frakCommandfactory#1{%
   \expandafter\def\csname frak#1\endcsname{\mathfrak{#1}}}

\newcounter{ctr}
\loop
  \stepcounter{ctr}
  \edef\X{\@Alph\c@ctr}
  \expandafter\calCommandfactory\X
  \expandafter\frakCommandfactory\X
  \edef\Y{\@alph\c@ctr}
  \expandafter\frakCommandfactory\Y
\ifnum\thectr<26
\repeat

\renewcommand{\cC}{\mathscr{C}}

\renewcommand{\cR}{\mathscr{R}}

\setenumerate{label={\normalfont (\roman*)},itemsep=0pt}

\def\lowfwd #1#2#3{{\mathop{\kern0pt #1}\limits^{\kern#2pt\raise.#3ex
\vbox to 0pt{\hbox{$\scriptscriptstyle\rightarrow$}\vss}}}}
\def\lowbkwd #1#2#3{{\mathop{\kern0pt #1}\limits^{\kern#2pt\raise.#3ex
\vbox to 0pt{\hbox{$\scriptscriptstyle\leftarrow$}\vss}}}}
\def\fwd #1#2{{\lowfwd{#1}{#2}{15}}}
\def\Sinf{S_{\aleph_0}}

\def\vE{{\hskip-1pt{\fwd{E}{3.5}}\hskip-1pt}}
\def\vF{{\hskip-1pt{\fwd{F}{3.5}}\hskip-1pt}}
\def\ve{\kern-1.5pt\lowfwd e{1.5}2\kern-1pt}
\def\ev{\kern-1pt\lowbkwd e{0.5}2\kern-1pt}
\def\vf{\kern-2pt\lowfwd f{2.5}2\kern-1pt}

\newtheorem{theorem}{Theorem}[section] 
\newtheorem{corollary}[theorem]{Corollary}
\newtheorem{lemma}[theorem]{Lemma}

\newtheorem{mainresult}{Theorem}

\newtheorem{innercustomthm}{Theorem}

\newenvironment{customthm}[1]
  {\renewcommand\theinnercustomthm{#1}\innercustomthm}
  {\endinnercustomthm}

\theoremstyle{definition}
\newtheorem{example}[theorem]{Example}

\theoremstyle{remark}

\begin{document}

\title[Duality theorems for stars and combs II: Dominated combs]{Duality theorems for stars and combs\\
II: Dominating stars and dominated combs}

\author{Carl Bürger}
\author{Jan Kurkofka}
\address{University of Hamburg, Department of Mathematics, Bundesstraße 55 (Geomatikum), 20146 Hamburg, Germany}
\email{carl.buerger@uni-hamburg.de, jan.kurkofka@uni-hamburg.de}


\@namedef{subjclassname@2020}{\textup{2020} Mathematics Subject Classification}

\begin{abstract}
In a series of four papers we determine structures whose existence is dual, in the sense of complementary, to the existence of stars or combs.
Here, in the second paper of the series, we present duality theorems for combinations of stars and combs: dominating stars and dominated combs.
As dominating stars exist if and only if dominated combs do, the structures complementary to them coincide.
Like for arbitrary stars and combs, our duality theorems for dominated combs (and dominating stars) are phrased in terms of normal trees or tree-decompositions.

The complementary structures we provide for dominated combs unify those for stars and combs and allow us to derive our duality theorems for stars and combs from those for dominated combs.
This is surprising given that our complementary structures for stars and combs are quite different: those for stars are locally finite whereas those for combs are~rayless.
\end{abstract}
\vspace*{-1.14cm} 
\maketitle

\vspace*{-.75cm}

\section{Introduction}
\noindent Two properties of infinite graphs are \emph{complementary} in a class of infinite graphs if they partition the class.
In a series of four papers we determine structures whose existence is complementary to the existence of two substructures that are particularly fundamental to the study of connectedness in infinite graphs: stars and combs.
See~\cite{StarComb1StarsAndCombs} for a comprehensive introduction, and a brief overview of results, for the entire series of four papers (\cite{StarComb1StarsAndCombs,StarComb3TheUndominatedComb,StarComb4TheUndominatingStar} and this paper).

In the first paper~\cite{StarComb1StarsAndCombs} of this series we found structures whose existence is complementary to the existence of a star or a comb \at a given set $U$ of vertices.
A~\emph{comb} is the union of a ray $R$ (the comb's \emph{spine}) with infinitely many disjoint finite paths, possibly trivial, that have precisely their first vertex on~$R$. 
The last vertices of those paths are the \emph{teeth} of this comb.
Given a vertex set $U$, a \emph{comb attached to} $U$ is a comb with all its teeth in $U$, and a \emph{star attached to} $U$ is a subdivided infinite star with all its leaves in $U$.
Then the set of teeth is the \emph{attachment set} of the comb, and the set of leaves is the \emph{attachment set} of the star.

As stars and combs can interact with each other, this is not the end of the story.
For example, a given vertex set $U$ might be connected in a graph $G$ by both a star and a comb, even with infinitely intersecting sets of leaves and teeth. 
To formalise this, let us say that a subdivided star $S$ \emph{dominates} a comb $C$ if infinitely many of the leaves of $S$ are also teeth of $C$.
A \emph{dominating star} in a graph~$G$ then is a subdivided star $S\subset G$ that dominates some comb $C\subset G$; and a \emph{dominated comb} in $G$ is a comb $C\subset G$ that is dominated by some subdivided star $S\subset G$.
In this second paper of our series we determine structures whose existence is complementary to the existence of  dominating stars and dominated combs.
Note that duality theorems for dominated combs are by nature also duality theorems for dominating stars, because for a graph $G$ and a vertex set $U\subset V(G)$ the existence of a dominated comb \at $U$ is equivalent to the existence of a dominating star \at ~$U$. 
For the sake of readability, we will state our duality theorems only for dominated combs.

Our first duality theorem for dominated combs is phrased in terms of normal trees.
A rooted tree $T\subset G$ is \emph{normal} in $G$ if the endvertices of every $T$-path in $G$ are comparable in the tree-order of $T$.
A vertex $v$ of $G$ \emph{dominates} a ray $R\subset G$ if there is an infinite $v$--$(R-v)$ fan in~$G$. 
For example, a comb is dominated in $G$ if and only if its spine is dominated in $G$.
Rays not dominated by any vertex are \emph{undominated}.
An end of $G$ is \emph{dominated} and \emph{undominated} if one (equivalently:~each) of its rays is dominated and undominated, respectively.
(See Diestel's textbook~\cite{DiestelBook5}.)

\begin{customthm}{\ref{domCombNTduality}}
\TFAD
\begin{enumerate}
    \item $G$ contains a \dblue{dominated comb} \at $U$;
    \item there is a \dblue{normal tree} $T\subset G$ that contains $U$ and all whose rays are undominated in $G$.
\end{enumerate}
Moreover, the normal tree $T$ in \emph{(ii)} can be chosen such that it contains $U$ cofinally and every component of $G-T$ has finite neighbourhood.
\end{customthm}

When a graph contains no star or no comb \at $U$, then in particular it contains no dominated comb \at $U$.
Hence, by our theorem, the graph contains a certain normal tree.
If there is no star, then this normal tree will be locally finite;
and if there is no comb, then it will be rayless.
Therefore, our duality theorem for dominated combs in terms of normal trees implies our duality theorems for arbitrary stars and combs in terms of normal trees from~\cite{StarComb1StarsAndCombs}, Theorems~\ref{thm: comb NT duality} and~\ref{thm: star NT duality}.
This is surprising given that infinite trees cannot be locally finite and rayless at the same time.

As an application, we will generalise Diestel's structural characterisation~\cite{VTopComp} of the graphs for which the topological spaces obtained by adding their ends are metrisable.
Depending on the topology chosen, Diestel characterised these graphs in terms of normal spanning trees, dominated combs, and infinite stars.
Applying Theorem~\ref{domCombNTduality}, 
we can now provide, for any given set $U$ of vertices, characterisations and simple existence criteria for connected metrisable (standard) subspaces containing $U$ in the various topologies.  
Our results will be in terms of normal trees containing $U$, dominated combs \at ~$U$, and stars \at ~$U$.

Theorem~\ref{domCombNTduality} is significantly strengthened by its `moreover' part.
It will be needed in the proof of our second 
duality theorem for dominated combs which is phrased in terms of tree-decompositions.
For the definition of tree-decompositions see~\cite{DiestelBook5}.
`Essentially disjoint' and `displaying' are defined in Section~\ref{section:domCombTDC}.
An end $\omega$ of a graph $G$ is contained \emph{in the closure} of a vertex set $U\subset V(G)$ in $G$ if $G$ contains a comb \at $U$ whose spine lies in~$\omega$.

\begin{customthm}{\ref{thm: domcomb tdc dual III}}
\TFAD
\begin{enumerate}
    \item $G$ contains a \dblue{dominated comb} \at $U$;
    \item $G$ has a rooted \dblue{tree-decomposition} $(T,\cV)$ such that:
    \begin{itemize}
        \item[\textbf{--}] each part contains at most finitely many vertices from $U$;
        \item[\textbf{--}]all parts at non-leaves of $T$ are finite;
        \item[\textbf{--}]$(T,\cV)$ has essentially disjoint connected adhesion sets;
        \item[\textbf{--}]$(T,\cV)$ displays the ends of $G$ in the closure of~$U$ in $G$.
    \end{itemize}
    \end{enumerate}
\end{customthm}

Similar to Theorem~\ref{domCombNTduality}, our duality theorem for dominated combs in terms of tree-decompositions implies our duality theorems for arbitrary stars and combs in terms of tree-decompositions from~\cite{StarComb1StarsAndCombs}, Theorems~\ref{thm: comb tree-decomposition duality} and~\ref{thm: star tree-decomposition duality}.

In our proof of Theorem~\ref{thm: domcomb tdc dual III} we employ a profound theorem of Carmesin~\cite{carmesin2014all}, which states that every graph has a tree-decomposition displaying all its undominated ends.
As it will be the case in this paper, Carmesin's theorem might often be used for graphs with normal spanning trees.
For this particular case we provide a substantially shorter~proof.

This paper is organised as follows.
Section~\ref{section:domCombNT} establishes our duality theorem for dominated combs in terms of normal trees.
In Section~\ref{section:domCombTDC} we prove our duality theorems for dominated combs in terms of tree-decompositions. Our short proof of Carmesin's theorem for graphs with a normal spanning tree can be found there as well.

Throughout this paper, $G=(V,E)$ is an arbitrary infinite graph.
We use the graph theoretic notation of Diestel's book~\cite{DiestelBook5}, and we assume familiarity with the tools and terminology described in the first paper of this series~\cite[Section~2]{StarComb1StarsAndCombs}.


\section{Dominated combs and normal trees}\label{section:domCombNT}

\noindent In this section we obtain the following duality theorem for dominated combs in terms of normal trees:

\begin{mainresult}\label{domCombNTduality}
\TFAD
\begin{enumerate}
    \item $G$ contains a dominated comb \at $U$;
    \item there is a normal tree $T\subset G$ that contains $U$ and all whose rays are undominated in $G$.
\end{enumerate}
Moreover, the normal tree $T$ in \emph{(ii)} can be chosen such that it contains $U$ cofinally and every component of $G-T$ has finite neighbourhood. 
\end{mainresult}

\noindent The inconspicuous `moreover' part will pave the way for our duality theorem for dominated combs in terms of tree-decompositions (Theorem~\ref{thm: domcomb tdc dual III}).

Before we provide a proof of Theorem~\ref{domCombNTduality} above, we shall discuss some consequences and applications. 
As a first consequence, Theorem~\ref{domCombNTduality} above builds a bridge between the duality theorems for combs (Theorem~\ref{thm: comb NT duality})
and stars (Theorem~\ref{thm: star NT duality})
in terms of normal trees, which we recall here. 
\begin{theorem}[{\cite[Theorem~1]{StarComb1StarsAndCombs}}]\label{thm: comb NT duality}
\TFAD
\begin{enumerate}
\item $G$ contains a comb \at $U$;
\item there is a rayless normal tree $T\subset G$ that contains $U$.
\end{enumerate}
Moreover, the normal tree $T$ in \emph{(ii)} can be chosen so that it contains $U$ cofinally.
\end{theorem}
\begin{theorem}[{\cite[Theorem~6]{StarComb1StarsAndCombs}}]\label{thm: star NT duality}
\TFAD
\begin{enumerate}
    \item $G$ contains a star \at $U$;
    \item there is a locally finite normal tree $T\subset G$ that contains $U$ and all whose rays are undominated in $G$.
\end{enumerate}
Moreover, the normal tree $T$ in \emph{(ii)} can be chosen such that it contains $U$ cofinally and every component of $G-T$ has finite neighbourhood.
\end{theorem}

Our duality theorem for dominated combs in terms of normal trees implies the corresponding duality theorems for combs and stars above.
This becomes apparent by a close look at Figure~\ref{fig: rel duality thms in terms of NTs}. 
The three columns of the diagram summarise the three duality theorems. 
Arrows depict implications between the statements; the dashed arrows indicate that further assumptions are needed to obtain their implications. 
On the left hand side, the extra assumption is that there is no comb \at ~$U$; on the right hand side, the extra assumption is that there is no star \at ~$U$.

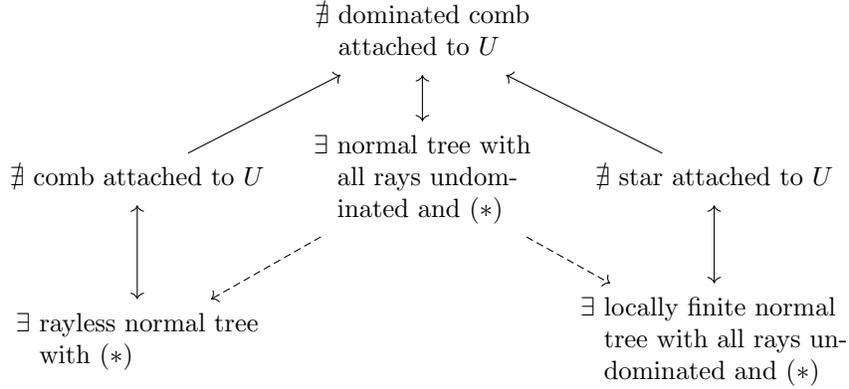
\begin{figure}[ht]
\begin{tikzcd}[column sep=0cm] 
& {
{\begin{tabular}{l}
$\nexists$ dominated comb \\ \phantom{$\nexists$ }\at~$U$\end{tabular}}
}\arrow[d,leftrightarrow]\\ 
{\begin{tabular}{l}
$\nexists$ comb \at~$U$\end{tabular}}\arrow[d,leftrightarrow]\arrow[ur] & {\begin{tabular}{l}
    $\exists$ normal tree with \\
    \phantom{$\exists$} all rays undom- \\ \phantom{$\exists$} inated and ($\ast$)
\end{tabular}}\arrow[dl,rightarrow,dashed]\arrow[dr,rightarrow,dashed] & {\begin{tabular}{l}
$\nexists$ star \at~$U$\end{tabular}}\arrow[ul,rightarrow]\arrow[d,leftrightarrow]\\
{\begin{tabular}{l}
    $\exists$ rayless normal tree\\\phantom{$\exists$} with ($\ast$)
\end{tabular}}& & {\begin{tabular}{l}
    $\exists$ locally finite normal\\
    \phantom{$\exists$} tree with all rays un-\\ \phantom{$\exists$} dominated and ($\ast$)
    \end{tabular}}
\end{tikzcd}
    \caption{The relations between the duality theorems for combs, stars and dominated combs in terms of normal trees.\newline
    Condition ($\ast$) says that the normal tree contains $U$ cofinally and every component of the graph minus the normal tree has finite neighbourhood.}
    \label{fig: rel duality thms in terms of NTs}
\end{figure}

\noindent As a consequence of the two dashed arrows, we obtain the implications $\neg$(i)$\to$(ii) of Theorem~\ref{thm: comb NT duality} and of Theorem~\ref{thm: star NT duality} from the corresponding implication of Theorem~\ref{domCombNTduality}.
Indeed, if $G$ does not contain a comb \at $U$, then in particular it does not contain a dominated comb \at $U$. 
Hence Theorem~\ref{domCombNTduality} yields a normal tree, which additionally must be rayless. 
Similarly, if $G$ does not contain a star \at $U$, then in particular it does not contain a dominated comb \at ~$U$.
Hence Theorem~\ref{domCombNTduality} yields a normal tree, which additionally must be locally finite and satisfy that all its rays are undominated. 
Since (i) and (ii) of Theorem~\ref{thm: comb NT duality} and of Theorem~\ref{thm: star NT duality} exclude each other almost immediately we have, so far, derived these two duality theorems for combs and stars from our duality theorem for dominated combs---except for the `moreover' part of Theorem~\ref{thm: star NT duality}.

We proved Theorem~\ref{thm: star NT duality} without its `moreover' part in the first paper~\cite{StarComb1StarsAndCombs} of our series.
There, instead of proving the `moreover' part as well, we announced that we would prove it in this second paper of the series.
And here we prove it, by deriving it from the identical `moreover' part of Theorem~\ref{domCombNTduality}:

\begin{proof}[{Proof of Theorem~\ref{thm: star NT duality}, including its `moreover' part}]
Employ Theorem~\ref{domCombNTduality} as above.
\end{proof}

Another consequence of Theorem~\ref{domCombNTduality} is a fact whose previous proof, \cite[Lemma~2.3]{VTopComp}, relied on the theorem of Halin~\cite{halin78} which states that every connected graph without a subdivided $K^{\aleph_0}$ has a normal spanning tree:

\begin{corollary}
If $G$ is a connected graph none of whose ends is dominated, then $G$ is normally spanned.\qed
\end{corollary}

For the proof of Theorem~\ref{domCombNTduality}, we shall need the following four lemmas and a result by Jung (cf.~\cite[Satz~6]{jung69} or~\cite[Theorem~3.5]{StarComb1StarsAndCombs}). 
The first lemma is from the first paper of this series and we remark that the original statement also takes critical vertex sets in the closure of $T$ or $W$ into account.

\begin{lemma}[{see \cite[Lemma~2.13]{StarComb1StarsAndCombs}}]\label{lemma: closure and tree containing U cofinally}
Let $G$ be any graph.
If $T\subset G$ is a rooted tree that contains a vertex set $W$ cofinally, then $\Abs{T}=\Abs{W}$.
\end{lemma}
Recall that for a graph $G$ and a normal tree $T\subset G$ the \emph{generalised up-closure} $\guc{x}$ of a vertex $x\in T$ is the union of $\uc{x}$ with the vertex set of $\bigcup \cC(x)$, where the set $\cC(x)$ consists of those components of $G-T$ whose neighbourhoods meet $\uc{x}$.
\begin{lemma}[{\cite[Lemma~2.10]{StarComb1StarsAndCombs}}]\label{lemma: separation properties normal tree}
Let $G$ be any graph and $T\subset G$ any normal tree.
\begin{enumerate}
    \item Any two vertices $x,y\in T$ are separated in $G$ by the vertex set $\dc{x}\cap\dc{y}$.
    \item Let $W\subset V(T)$ be down-closed. Then the components of $G-W$ come in two types: the components that avoid $T$; and the components that meet $T$, which are spanned by the 
    sets $\guc{x}$ with $x$ minimal in $T-W$. 
\end{enumerate}
\end{lemma}

\begin{lemma}[{\cite[Lemma~2.11]{StarComb1StarsAndCombs}}]\label{lemma: nts reflect ends in closure}
If $G$ is any graph and $T\subset G$ is any normal tree,
then every end of $G$ in the closure of $T$ contains exactly one normal ray of $T$.
Moreover, sending these ends to the normal rays they contain defines a bijection between $\Abs{T}$ and the normal rays of~$T$.
\end{lemma}

\begin{lemma}\label{distanceClassesDispersed}
Let $G$ be a connected graph, let $D_0,D_1,\ldots$ be the distance classes of $G$ with respect to an arbitrary vertex of $G$, and let $n\ge 1$.
Then for every infinite $U\subset D_n$ the induced subgraph $G[D_0\cup\cdots\cup D_n]$ contains a star \at $U$.
\end{lemma}

\begin{proof}
Consider any spanning tree $T$ of $G[D_0\cup\cdots \cup D_n]$ whose $k$th level is equal to $D_k$ for all~$k\le n$.
As $T$ is rayless, it contains a star \at~$U$.
\end{proof}

\begin{theorem}[Jung]\label{thm: NT and dispersed sets}
Let $G$ be any graph. A vertex set $W\subset V(G)$ is normally spanned if and only if it is a countable union of dispersed sets. In particular, 
$G$ is normally spanned if and only if $V(G)$ is a countable union of dispersed sets.
\end{theorem}

Now we are ready to prove our first duality theorem for dominated combs:

\begin{proof}[Proof of Theorem~\ref{domCombNTduality}]
First, we show that at most one of (i) and (ii) holds.
Assume for a contradiction that both hold, let $R$ be the spine of a dominated comb \at $U$ and let $T$ be a normal tree as in (ii).
Then the end of $R$ lies in the closure of $U\subset T$, so by Lemma~\ref{lemma: nts reflect ends in closure} the normal tree $T$ contains a normal ray from that end.
But then the vertices dominating $R$ in $G$ also dominate that normal ray, a contradiction.

It remains to show that at least one of (i) and (ii) holds; we show $\neg$(i)$\to$(ii).
For this, let $\cR$ be an inclusionwise maximal collection of pairwise disjoint rays all belonging to ends in the closure of~$U$, and define $\hat{U}:=U\cup\bigcup\,\{\,V(R)\mid R\in\cR\,\}$.
We claim that $\Abs{\hat{U}}=\Abs{U}$.
Clearly, $\hat{U}\supset U$ implies $\Abs{\hat{U}}\supset\Abs{U}$.
For the forward inclusion, we show that every end $\omega$ that does not lie in the closure $U$ also does not lie in the closure of~$\hat{U}$.
Indeed, consider any finite vertex set $X\subset V(G)$ such that the component $C(X,\omega)$ avoids~$U$.
Since the rays in~$\cR$ are pairwise disjoint and belong to ends in the closure of~$U$, the intersection $C(X,\omega)\cap\bigcup\cR$ is finite.
Hence we may extend~$X$ to ensure that $C(X,\omega)$ avoids~$\hat{U}$, showing that $\omega$ does not lie in the closure of~$\hat{U}$ as desired.

Note that every end in the closure of~$U$ is undominated since there is no dominated comb \at ~$U$, and hence every end in the closure of~$\hat{U}$ is undominated as well.

Next, we find a normal tree $T\subset G$ that contains~$\hat{U}$, as follows.
We pick an arbitrary vertex $v_0$ of $G$ and write $D_n$ for the $n$th distance class of $G$ with respect to $v_0$.
If for some distance class $D_n$ there was a comb in $G$ \at $D_n\cap \hat{U}$, then that comb would be dominated by Lemma~\ref{distanceClassesDispersed} contrary to our assumptions.
Therefore, all the sets $D_n\cap \hat{U}$ with $n\in\N$ are dispersed. 
Now, Jung's Theorem~\ref{thm: NT and dispersed sets} yields a normal tree $\hat{T}\subset G$ that contains $\hat{U}$, and by replacing $\hat{T}$ with the down-closure of $\hat{U}$ we may assume that $\hat{T}$ contains $\hat{U}$ cofinally.
Note that the normal rays of $\hat{T}$ cannot be dominated in $G$ because $\Abs{\hat{T}}=\Abs{\hat{U}}=\Abs{U}$ by Lemma~\ref{lemma: closure and tree containing U cofinally}.

We claim that every component $C$ of $G-\hat{T}$ has finite neighbourhood.
For this, assume for a contradiction that some component $C$ of $G-\hat{T}$ has infinite neighbourhood.
Let $R$ be the normal ray in $\hat{T}$ given by the down-closure of that neighbourhood in $\hat{T}$, and write $Z$ for the set of those vertices in $C$ that send edges to $\hat{T}$. 
Since $R$ is undominated in~$G$, every vertex in $Z$ may send only finitely many edges to $R$, and in particular $Z$ must be infinite.
Therefore, we find an infinite subset $Z'\subset Z$ for which $G$ contains a matching of $Z'$ and an infinite subset of $V(R)$.
Applying the star-comb lemma in $C$ to $Z'$ then, as $R$ was just noted to be undominated, must yield a comb in $C$ \at $Z'$.
That comb's spine $R'$ is equivalent in $G$ to~$R$ and avoids $\hat{U}$, contradicting the choice of $\hat{U}$.

Finally, let $T\subset G$ be the normal tree given by the down-closure of $U$ in $\hat{T}$.
Then $T$ contains $U$ cofinally.
We claim that every component of $G-T$ has a~finite neighbourhood.
Indeed, consider any component $C$ of $G-T$. 
If $C$ is also a component of $G-\hat{T}$, then---as we have already seen---it has a finite neighbourhood. 
Otherwise, by Lemma~\ref{lemma: separation properties normal tree}, the component $C$ is spanned by $\guc{x}$ with respect to $\hat{T}$ for the minimal node $x$ in $C\cap \hat{T}$.
Now, as $\hat{T}$ is normal, $C$ can only send edges to the finite set $\dc{x}\setminus\{x\}$. Hence the component $C$ has finite neighbourhood as claimed.
\end{proof}

Let us discuss an application of our duality theorem for dominated combs in terms of normal trees. 
In \cite{VTopComp}, Diestel proves the following theorem that relates the metrisability of $|G|$ to the existence of normal spanning trees (we refer to \cite[Section~2]{VTopComp} for definitions concerning $\vert G\vert$, \MTop, \VTop\ and \Top): 

\begin{theorem}[{\cite[Theorem~3.1]{VTopComp}}]\label{thm: metrisability and NST}
Let $G$ be any connected graph. \begin{enumerate}
    \item In \MTop, $|G|$ is metrisable if and only if $G$ has a normal spanning tree. 
    \item In \VTop, $|G|$ is metrisable if and only if no end of $G$ is dominated. 
    \item In \Top, $|G|$ is metrisable if and only if $G$ is locally finite. 
\end{enumerate}
\end{theorem}
\noindent Assertions (ii) and (iii) of this theorem can be reformulated so as to speak about normal spanning trees: 
By Theorem~\ref{domCombNTduality} with $U=V(G)$, the graph $G$ having no dominated end is equivalent to $G$ having a normal spanning tree all of whose normal rays are undominated. 
And by Theorem~\ref{thm: star NT duality} with $U=V(G)$, the graph $G$ being locally finite is equivalent to $G$ having a locally finite normal spanning tree all of whose normal rays are undominated.
That is why we may hope that these theorems allow us to localise Theorem~\ref{thm: metrisability and NST} above to arbitrary vertex sets $U\subseteq V(G)$.
We will show that this is possible.

Recall that a \emph{standard subspace} of $|G|$ (with regard to \MTop, \VTop\ or \Top) is a subspace $Y$ of $|G|$ that is the closure $\overline{H}$ of a subgraph $H$ of $G$ (see Diestel's textbook~\cite[p.\,246]{DiestelBook5}).

\begin{lemma}\label{lemma: normal tree SS coincides with T with tops}
Let $G$ be any graph, let $T\subseteq G$ be any normal tree and consider the spaces $|T|$ and $|G|$, both in the same choice of one of the three topologies \MTop, \VTop\ or \Top. Then $|T|$ is homeomorphic to the standard subspace $\overline{T}$ of $|G|$.
\end{lemma}

\begin{proof}
By Lemma~\ref{lemma: nts reflect ends in closure}, the identity on $T$ extends to a bijection $\vert T\vert\to\overline{T}\subset\vert G\vert$ that sends every end of $T$ to the unique end of $G$ including it.
Using Lemma~\ref{lemma: separation properties normal tree} it is straightforward to verify that the bijection is a homeomorphism, no matter which of the three topologies we chose.
\end{proof}

\begin{lemma}\label{lemma: conn metr sss implies nt}
Let $G$ be any connected graph and $U\subseteq V(G)$ any vertex set. If the space $\vert G\vert$ with one of the three topologies \MTop, \VTop\ or \Top\ has a connected metrisable standard subspace containing $U$, then there is a normal tree $T\subset G$ that contains $U$. 
\end{lemma}

\begin{proof}
We imitate Diestel's proof of the corresponding implication  of his Theorem~\ref{thm: metrisability and NST}~(i). 
Recall from \cite{VTopComp} that a set of vertices of $G$ is dispersed in $G$ if and only it is closed in $\vert G\vert$. So by Jung's Theorem~\ref{thm: NT and dispersed sets}, it suffices to show that $U$ can be written as a countable union of closed vertex sets. For this, the sets $U_n$ consisting of the vertices in $U$ that have distance $\ge 1/n$ from every end can be taken: 
On the one hand, every $U_n$ is the intersection of complements of open balls of radius $1/n$, and hence closed. 
On the other hand, every vertex $u\in U$ is contained in $U_n$ for some $n\in \N$ because $G$ is open in $\vert G\vert$.
\end{proof}

\begin{theorem}\label{thm: metrisability and NT localised}
Let $G$ be any connected graph and $U\subseteq V(G)$ any vertex set. \begin{enumerate}
    \item In \textsc{MTop} or \textsc{VTop}, $\vert G\vert$ has a connected metrisable standard subspace containing $U$ if and only if there is a normal tree $T\subset G$ that contains $U$. 
    In particular, if there is no dominated comb \at $U$, then $\vert G\vert$ has a connected metrisable standard subspace containing $U$.
    \item In \textsc{Top}, $\vert G\vert$ has a connected metrisable standard subspace containing $U$ if and only if there is a locally finite normal tree $T\subset G$ that contains $U$. 
    In particular, if there is no star \at $U$, then $\vert G\vert$ has a connected metrisable standard subspace containing $U$.
\end{enumerate}
\end{theorem}

\begin{proof}
(i) The forward implication is covered by Lemma~\ref{lemma: conn metr sss implies nt}. Now, suppose that there is a normal tree $T\subset G$ containing $U$ and consider the standard subspace $\overline{T}$.
By Lemma~\ref{lemma: normal tree SS coincides with T with tops} the spaces $\overline{T}$ and $|T|$ are homeomorphic. Then $\vert T\vert$ is metrisable by Theorem~\ref{thm: metrisability and NST}~(i) respectively~(ii).
The `in particular' part is a consequence of Theorem~\ref{domCombNTduality}.

(ii) For the forward implication we apply Lemma~\ref{lemma: conn metr sss implies nt} to a given standard subspace $\overline{H}$ and $U\subset H$ to obtain a normal tree $T$ that contains~$U$. By taking the down-closure of $U$ in $T$ we may assume that $T$ contains $U$ cofinally. We claim that $T$ is locally finite. Indeed, if $T$ is not locally finite then there is a finite vertex set $X$ of $G$ such that infinitely many components of $G-X$ meet~$U$.
Since $\overline{H}$ is connected, the subgraph $H$ of $G$ must contain an edge between each component and $X$.
By the pigeonhole principle, some vertex of $X$ has infinite degree in~$H$, contradicting the fact that $\overline{H}$ is metrisable.

For the backward implication, suppose that there is a locally finite normal tree $T\subseteq G$ containing~$U$. By Lemma~\ref{lemma: normal tree SS coincides with T with tops} we have that the standard subspace that arises from $T$ is homeomorphic to $|T|$ with \Top. Since $T$ is locally finite, \Top\ coincides with \MTop\ on $|T|$ which is metrisable by Theorem~\ref{thm: metrisability and NST}~(i). 

The `in particular' part is a consequence of Theorem~\ref{thm: star NT duality}.
\end{proof}



\section{Dominated combs and tree-decompositions}\label{section:domCombTDC}

\noindent In the previous section, we have presented a duality theorem for dominated combs in terms of normal trees. 
And we have deduced from this theorem the hard implications $\neg$(i)$\to$(ii) of Theorem~\ref{thm: comb NT duality} and of Theorem~\ref{thm: star NT duality} (the duality theorems for combs and stars in terms of normal trees).

Therefore we may expect from a duality theorem for dominated combs in terms of tree-decompositions to reestablish the hard implications $\neg$(i)$\to$(ii) of the duality theorems for combs and stars in terms of tree-decompositions (Theorem~\ref{thm: comb tree-decomposition duality} and Theorem~\ref{thm: star tree-decomposition duality} below)---by following arrows in Figure~\ref{fig: desired rel duality thms in terms of tdcs} like we did in Figure~\ref{fig: rel duality thms in terms of NTs}. 
\begin{theorem}[{\cite[Theorem~2]{StarComb1StarsAndCombs}}]\label{thm: comb tree-decomposition duality}
\TFAD
\begin{enumerate}
    \item $G$ contains a comb \at $U$;
    \item $G$ has a rayless tree-decomposition into parts each containing at most finitely many vertices from $U$ and whose parts at non-leaves of the decomposition tree are all finite.
\end{enumerate}
Moreover, the tree-decomposition in \emph{(ii)} can be chosen with connected separators.
\end{theorem}

Recall from~\cite{StarComb1StarsAndCombs} that a tree-decomposition $(T,\cV)$ of a given graph $G$ with finite separators \emph{displays} a set $\Psi$ of ends of $G$ if $\tau$ restricts to a bijection $\tau\rest\Psi\colon\Psi\to\Omega(T)$ between $\Psi$ and the end space of $T$  and maps every end that is not contained in $\Psi$ to some node of $T$, where $\tau\colon\Omega(G)\to\Omega(T)\sqcup V(T)$ maps every end of $G$ to the end or node of $T$ which it corresponds to or lives at, respectively.
\begin{theorem}[{\cite[Theorem~7]{StarComb1StarsAndCombs}}]
\label{thm: star tree-decomposition duality}
\TFAD
\begin{enumerate}
\item $G$ contains a star \at $U$;
\item $G$ has a locally finite tree-decomposition with finite and pairwise disjoint separators such that each part contains at most finitely many vertices of $U$.
\end{enumerate}
Moreover, the tree-decomposition in \emph{(ii)} can be chosen with connected separators and so that it displays $\Abs{U}$. 
\end{theorem}

In Section~\ref{subsection:domCombTreeForComb}, we will prove a duality theorem for dominated combs in terms of tree-decompositions, making the left but not the right dashed arrow in Figure~\ref{fig: desired rel duality thms in terms of tdcs} true.
In Section~\ref{subsection:domCombTreeForStar}, the situation is reversed: we will prove a duality theorem 
making the right but not the left dashed arrow in Figure~\ref{fig: desired rel duality thms in terms of tdcs} true. 
Here we also provide a short proof of Carmesin's result~\cite{carmesin2014all}, which states that every graph has a tree-decomposition displaying all its undominated ends, for normally spanned graphs. 
Finally, in Section~\ref{subsection: domcomb combined}, we will prove a duality theorem that makes both the left and the right dashed arrow in Figure~\ref{fig: desired rel duality thms in terms of tdcs} true. 
This will be achieved by combining our proof techniques from Section~\ref{subsection:domCombTreeForComb} and Section~\ref{subsection:domCombTreeForStar}.

\begin{figure}[ht]
\begin{tikzcd}[column sep=0cm] 
& {\begin{tabular}{l} 
$\nexists$ dominated comb \\ \phantom{$\nexists$ }\at $U$\end{tabular}}\arrow[d,leftrightarrow]\\ 
{\begin{tabular}{l} 
$\nexists$ comb attached \\ \phantom{$\nexists$} to $U$ \end{tabular}}\arrow[d,leftrightarrow]\arrow[ur] & 
{\begin{tabular}{l} 
    \large{\textbf{?}}
\end{tabular}}\arrow[dl,rightarrow,dashed]\arrow[dr,rightarrow,dashed] & {\begin{tabular}{l} 
$\nexists$ star attached \\ \phantom{$\nexists$} to $U$\end{tabular}}\arrow[ul,rightarrow]\arrow[d,leftrightarrow]\\
{\begin{tabular}{l} 
$\exists$ complementary \\ \phantom{$\exists$} rayless tree-\\ \phantom{$\exists$ }decomposition\end{tabular}
} & & {\begin{tabular}{l} 
$\exists$ complementary \\ \phantom{$\exists$} locally finite tree-\\ \phantom{$\exists$} decomposition \end{tabular}}
\end{tikzcd}
    \caption{The desired relation between stars, combs, dominated combs and complementary tree-decompositions. \\
    The left and right dashed arrow describe an implication whenever there is no comb and no star \at~$U$, respectively.}
    \label{fig: desired rel duality thms in terms of tdcs}
\end{figure}
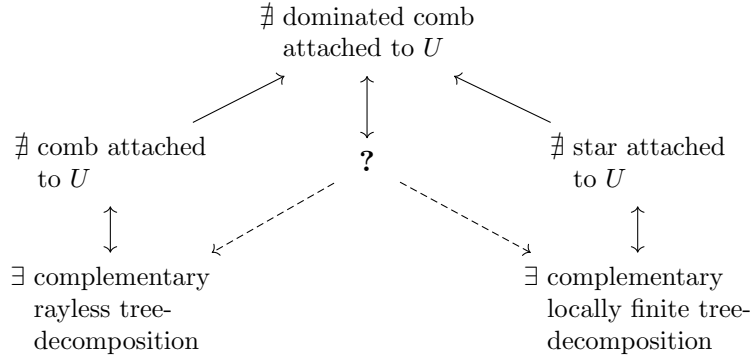


\subsection{A duality theorem related to combs}\label{subsection:domCombTreeForComb}
Here we present a duality theorem for dominated combs in terms of tree-decompositions making the left but not the right dashed arrow of Figure~\ref{fig: desired rel duality thms in terms of tdcs} true:
\begin{theorem}\label{domCombTreeDuality}
\TFAD
\begin{enumerate}
    \item $G$ contains a dominated comb \at $U$;
    \item $G$ has a tree-decomposition $(T,\cV)$ that satisfies:
    \begin{enumerate}[label=\normalfont (\alph*)]
        \item each part contains at most finitely many vertices from $U$;
        \item all parts at non-leaves of $T$ are finite;
        \item every dominated end of $G$ lives in a part at a leaf of $T$.
    \end{enumerate}
\end{enumerate}
Moreover, the tree-decomposition in \emph{(ii)} can be chosen with connected separators and so that it displays $\Abs{U}$.
\end{theorem} 

Before we provide a proof of this theorem, let us deduce the left dashed arrow of Figure~\ref{fig: desired rel duality thms in terms of tdcs} from it (also see Figure~\ref{fig: domcom tdc dual I} which shows the first two columns of Figure~\ref{fig: desired rel duality thms in terms of tdcs}  in greater detail and with Theorem~\ref{domCombTreeDuality}~(ii) including the theorem's `moreover' part inserted for `?'): If $G$ does not contain a comb \at $U$, then in particular it does not contain a dominated comb \at $U$. Hence Theorem~\ref{domCombTreeDuality} returns a tree-decomposition $(T,\cV)$ of $G$ which we may choose so that it satisfies the theorem's `moreover' part; in particular $(T,\cV)$ displays $\Abs{U}$.
Our assumption that there is no comb \at $U$ implies that $\Abs{U}$ is empty and hence $T$ is rayless.
Using the corresponding conditions from Theorem~\ref{domCombTreeDuality}~(ii) including the theorem's `moreover' part, we conclude that $(T,\cV)$ is as in Theorem~\ref{thm: comb tree-decomposition duality}~(ii) including the theorem's `moreover' part.

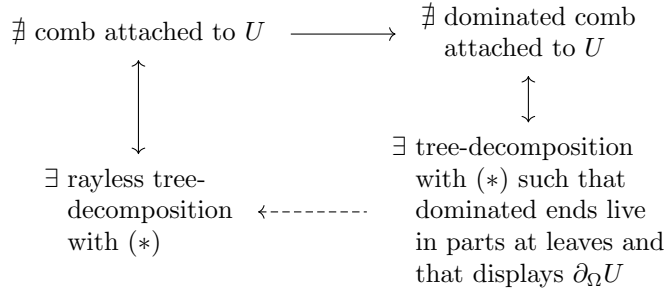
\begin{figure}[ht] 
\begin{tikzcd}[column sep=0cm]
{\begin{tabular}{l}
$\nexists$ comb \at~$U$\end{tabular}}\arrow[rrrr,rightarrow]\arrow[d,leftrightarrow]  & \text{}& \text{}& \text{}&  {\begin{tabular}{l}
$\nexists$ dominated comb \\ \phantom{$\nexists$ }\at~$U$\end{tabular}} \arrow[d,leftrightarrow] \\ 
{\begin{tabular}{l}
$\exists$ rayless tree-\\ \phantom{$\exists$ }decomposition \\
\phantom{$\exists$ }with ($\ast$)\end{tabular}
}\arrow[rrrr,leftarrow,dashed]  & \text{}& \text{}& \text{ }&  {\begin{tabular}{l}
$\exists$ tree-decomposition \\
\phantom{$\exists$ }with ($\ast$) such that \\
\phantom{$\exists$ }dominated ends live\\ 
\phantom{$\exists$} in parts at leaves and\\ 
\phantom{$\exists$} that displays $\Abs{U}$ \end{tabular}
} 
\end{tikzcd}
    \caption{The first two columns of Figure~\ref{fig: desired rel duality thms in terms of tdcs} with Theorem~\ref{domCombTreeDuality}~(ii) including the theorem's `moreover' part inserted for `?'.\\
    Condition ($\ast$) says that parts contain at most finitely many vertices from $U$, that parts at non-leaves are finite and that the separators are connected.}
    \label{fig: domcom tdc dual I}
\end{figure}

Finally, we prove Theorem~\ref{domCombTreeDuality}:
\begin{proof}[Proof of Theorem~\ref{domCombTreeDuality}]
First, we show that at most one of (i) and (ii) holds.
Assume for a contradiction that $G$ contains a dominated comb \at $U$ and has, at the same time, a tree-decomposition $(T,\cV)$ as in (ii).
Let $R$ be the comb's spine.
Since every dominated end of $G$ lives in a part at a leaf of $T$, and since all parts at non-leaves are finite, we find without loss of generality a leaf $\ell$ of $T$ with $R\subset G[V_\ell]$.
But each part contains at most finitely many vertices from $U$.
In particular, $V_\ell$ contains at most finitely many vertices from~$U$.
Therefore, the comb must send some infinitely many pairwise disjoint paths to vertices in $U\setminus V_\ell$.
But the separator of $G$ that is associated with the edge $\ell t\in T$ at $\ell$ is contained in the intersection $V_\ell\cap V_t\subset V_t$ which is finite since $V_t$ is, a contradiction.

Now, to show that at least one of (i) and (ii) holds, we show $\neg$(i)$\to$(ii).
By Theorem~\ref{domCombNTduality} we find a normal tree $\nt\subset G$ containing $U$ cofinally all whose rays are undominated in $G$ and such that every component of $G-\nt$ has finite neighbourhood.
We construct the desired tree-decomposition from $\nt$.

Given a component $C$ of $G-\nt$ the neighbourhood of $C$ is a finite chain in the tree-order of $\nt$, and hence has a maximal element $t_C\in \nt$. We obtain the tree $T$ from $\nt$ by adding each component $C$ of $G-\nt$ as a new vertex and joining it precisely to $t_C$.

Having defined the decomposition tree $T$ it remains to define the parts of the desired tree-decomposition. For nodes $t\in \nt\subseteq T$ we let $V_t$ consist of the down-closure $\dc{t}_{\nt}$ of $t$ in the normal tree $\nt$. And for newly added nodes $C$ we let $V_C$ be the union of $V_{t_C}$ and the vertex set of the component $C$, i.e., we put $V_C:=\dc{t_C}_{\nt}\cup V(C)$.  

Since $\nt$ is normal and contains $U$ cofinally, it follows by standard arguments employing Lemma~\ref{lemma: closure and tree containing U cofinally} and Lemma~\ref{lemma: nts reflect ends in closure} that $(T,\cV)$ displays $\Abs{U}$.
Conditions (a) and (b) hold by construction.
Combining (b) with $(T,\cV)$ displaying $\Abs{U}$ gives (c), which in turn is---as the rest of the `moreover' part---a~direct consequence of how the parts are defined.
\end{proof}

\begin{example}
The tree-decomposition in  Theorem~\ref{domCombTreeDuality}~(ii) cannot be chosen to additionally have pairwise disjoint separators, which shows that the theorem does not make the right dashed arrow in Figure~\ref{fig: desired rel duality thms in terms of tdcs} true. 
To see this suppose that $G$ consists of the first three levels of $T_{\aleph_0}$, the tree all whose vertices have countably infinite degree, and let $U=V(G)$. 
Then $G$ contains no comb \at $U$.
Suppose for a contradiction that $G$ has a tree-decomposition $(T,\cV)$ as in Theorem~\ref{domCombTreeDuality}~(ii) which additionally has pairwise disjoint separators. The graph $G$ being rayless and $U$ being the whole vertex set of $G$ together with our assumption that $(T,\cV)$ has pairwise disjoint separators makes sure that $(T,\cV)$ also displays $\Abs{U}$. 
In particular, by our argumentation in the text below Theorem~\ref{domCombTreeDuality},  $(T,\cV)$ is also a tree-decomposition of $G$ complementary to combs as in Theorem~\ref{thm: comb tree-decomposition duality}.
But then $(T,\cV)$ cannot have pairwise disjoint separators, as pointed out in \cite[Example~3.7]{StarComb1StarsAndCombs}.
\end{example}

\subsection{A duality theorem related to stars}\label{subsection:domCombTreeForStar}
Here we present a duality theorem for dominated combs in terms of tree-decompositions  making the right but not the left dashed arrow in Figure~\ref{fig: desired rel duality thms in terms of tdcs} true.

\begin{theorem}\label{thm: domcomb tdc dual II}
\TFAD
\begin{enumerate}
    \item $G$ contains a dominated comb \at $U$;
    \item $G$ has a rooted tree-decomposition with upwards disjoint finite separators that displays $\Abs{U}$.
\end{enumerate}
Moreover, the tree-decomposition in \emph{(ii)} can be chosen with connected separators and so that it covers $U$ cofinally.
\end{theorem}
Before we prepare the proof of our theorem, let us deduce the right dashed arrow of Figure~\ref{fig: desired rel duality thms in terms of tdcs} from it (also see Figure~\ref{fig: domcom tdc dual II} which shows  the last two columns of Figure~\ref{fig: desired rel duality thms in terms of tdcs}  in greater detail and where Theorem~\ref{thm: domcomb tdc dual II}~(ii) including the theorem's `moreover' part is inserted for `?'):
If $G$ does not contain a star \at $U$, then in particular it does not contain a dominated comb \at $U$. 
Hence Theorem~\ref{thm: domcomb tdc dual II} yields a rooted tree-decomposition $(T,\cV)$ of $G$ which we choose so that it also satisfies the theorem's `moreover' part; in particular $(T,\cV)$ covers $U$ cofinally.
By assumption, the star-comb lemma yields a comb in $G$ \at $U'$ for every infinite subset $U'$ of $U$.
Since $(T,\cV)$ displays $\Abs{U}$ this means that no part can meet $U$ infinitely.

We claim that $T$ must be locally finite.
To see this, suppose 
for a contradiction that $t\in T$ is a vertex of infinite degree. For every up-neighbour $t'$ of $t$ we choose a vertex from $U$ that is contained in a part $V_{t''}$ with $t''\ge t'$ in $T$.
As the part $V_t$ contains only finitely many vertices from $U$, all but finitely many of the chosen vertices are not contained in~$V_t$.
Then applying the star-comb lemma in $G$ to the infinitely many chosen vertices from $U$ yields a comb.
The end of the comb's spine must then live at $t$ because the separators of $(T,\cV)$ are all finite. 
But this contradicts the fact that $(T,\cV)$ displays $\Abs{U}$ which contains the end of the comb's spine.

Finally, it remains to show that the separators of $(T,\cV)$ are pairwise disjoint, but they need not be.
However, we can define an equivalence relation $\sim$ on the nodes of $T$ by declaring $t_1$ and $t_2$ to be equivalent if they have a common predecessor $s$ such that the separators associated with the edges $s t_1$ and $s t_2$ meet.
Then all equivalence classes are finite because $T$ is locally finite, and we may let $(T',\cV')$ be the tree-decomposition where $T'$ is obtained from $T$ by collapsing each equivalence class to a single vertex, and $\cV'=(\,V_C\mid C\in V(T)/{\sim}\,)$ with $V_C:=\bigcup\,\{\,V_t\mid t\in C\,\}$.
Notably, the separators of $(T',\cV')$ are pairwise disjoint and connected while $(T',\cV')$ still displays $\Abs{U}$ and cofinally covers~$U$, completing the argumentation.


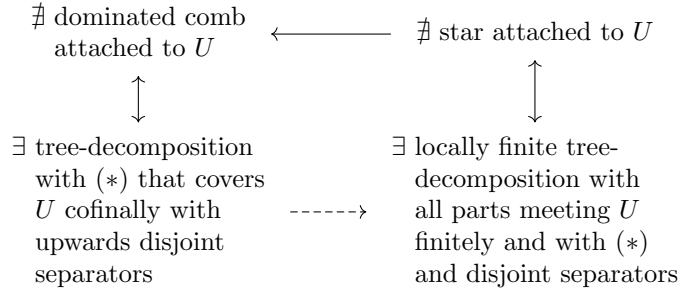
\begin{figure}[ht]
\begin{tikzcd}[column sep=0cm]
{\begin{tabular}{l}
$\nexists$ dominated comb \\ \phantom{$\nexists$ }\at~$U$\end{tabular}}\arrow[rrrr,leftarrow]\arrow[d,leftrightarrow]  & \text{}& \text{}& \text{}&  {\begin{tabular}{l}
$\nexists$ star \at~$U$\end{tabular}} \arrow[d,leftrightarrow] \\ 
{\begin{tabular}{l}
$\exists$ tree-decomposition \\
\phantom{$\exists$ }with ($\ast$) that covers\\
\phantom{$\exists$ }$U$ cofinally with\\
\phantom{$\exists$ }upwards disjoint \\ 
\phantom{$\exists$ }separators
\end{tabular}
} 
\arrow[rrrr,rightarrow,dashed]  & \text{}& \text{}& \text{ }&  {\begin{tabular}{l}
$\exists$ locally finite tree-\\ \phantom{$\exists$} decomposition with\\
\phantom{$\exists$} all parts meeting $U$\\ \phantom{$\exists$} finitely and with ($\ast$)\\ \phantom{$\exists$} and disjoint separators \end{tabular}
}
\end{tikzcd}
    \caption{The last two columns of Figure~\ref{fig: desired rel duality thms in terms of tdcs} with Theorem~\ref{thm: domcomb tdc dual II}~(ii) including the theorem's `moreover' part inserted for `?'. \newline
    Condition ($\ast$) says that the tree-decomposition displays $\Abs{U}$ and has finite connected separators. }
    \label{fig: domcom tdc dual II}
\end{figure}

In order to prove Theorem~\ref{thm: domcomb tdc dual II}, we will employ the following result by Carmesin. 
Recall that a rooted $\Sinf$-tree $(T,\alpha)$ has \emph{upwards disjoint} separators if 
for every two edges $\ve<\vf$ pointing away from the root $r$ of $T$ the separators of $\alpha(\ve)$ and $\alpha(\vf)$ are disjoint.
And $(T,\alpha)$ is \emph{upwards connected} if for every edge $\ve$ pointing away from the root $r$ the induced subgraph $G[B]$ stemming from $(A,B)=\alpha(\ve)$ is connected. 
A rooted tree-decomposition has \emph{upwards disjoint} separators or is \emph{upwards connected} if its corresponding $\Sinf$-tree is.

\begin{theorem}[{Carmesin 2014, \cite[Theorem~2.17]{StarComb1StarsAndCombs}}]\label{thm: tdc displaying undom ends}
Every connected graph $G$ has a rooted tree-decomposition with upwards disjoint finite connected separators that displays the undominated ends of $G$.
\end{theorem} 

 Carmesin's proof of this theorem in \cite{carmesin2014all} is long and complex. 
However, in this paper we need his theorem only for normally spanned graphs. This is why we will provide a substantially shorter proof for this class of graphs (cf.~Theorem~\ref{NormalTreeGivesCarmesinTDC}).

\begin{lemma}[{\cite[Lemma~2.16]{StarComb1StarsAndCombs}}]\label{lemma: tdc displays ends of G corrisponding to ends of T}
Let $G$ be any graph.
Every upwards connected rooted $\Sinf$-tree $(T,\alpha)$ with upwards disjoint separators displays the ends of $G$ that correspond to the ends of~$T$.
\end{lemma}

\begin{theorem}\label{NormalTreeGivesCarmesinTDC}
Let $G$ be any connected graph. If $\nt\subset G$ is a normal tree such that every component of $G-\nt$ has finite neighbourhood, 
then $G$ has a rooted tree-decomposition $(T,\cV)$ with the following three properties:
\begin{itemize}
    \item the separators are upwards disjoint, finite and connected;
    \item $(T,\cV)$ displays the undominated ends in the closure of $\nt$;
    \item $(T,\cV)$ covers $V(\nt)$ cofinally.
\end{itemize}
\end{theorem}

\begin{proof}
%
Let us write $r$ for the root of $\nt$. Recall that every component of $G-\nt$ has finite neighbourhood by assumption.
Hence every end $\omega\in \Omega\setminus \Abs{\nt}$ lives in a unique component of $G-\nt$; we define the \emph{height} of $\omega$ to be the height of the maximal neighbour of this component in $\nt$.

Starting with $T_0=r$ and $\alpha_0=\emptyset$ we recursively construct an ascending\footnote{Here, we mean ascending in both entries with regard to inclusion, i.e., $T_n\subset T_{n+1}$ and $\alpha_n\subset\alpha_{n+1}$ for all $n\in\N$.} sequence of  $S_{\aleph_0}$-trees $(T_n,\alpha_n)$ all rooted in $r$ and satisfying the following conditions: \begin{enumerate}
    \item the separators of $(T_n,\alpha_n)$ are upwards disjoint and they are vertex sets of ascending paths in $\nt$; 
    \item $T_{n}$ arises from $T_{n-1}$ by adding edges to its $(n-1)$th level;
    \item undominated ends in the closure of $\nt$ live at nodes of the $n$th level of $T_n$ with regard to $(T_n,\alpha_n)$;
    \item if $\omega\in \Omega\setminus \Abs{\nt}$ has height $<n$, then $\omega$ lives at a node of $T_n$ of height $<n$ with regard to $(T_n,\alpha_n)$. 
\end{enumerate}

Before pointing out the details of our construction, let us see how to complete the proof once the $(T_n,\alpha_n)$ are defined. Consider the $S_{\aleph_0}$-tree $(T,\alpha)$ defined by letting $T:=\bigcup_{n\in\N}T_n$ and  $\alpha:=\bigcup_{n\in\N}\alpha_n$, and let $(T,\cV)$ be the corresponding tree-decomposition of $G$.
By (i) we have that $(T,\cV)$ is indeed a rooted tree-decomposition with upwards disjoint finite connected separators all of which meet $V(\nt)$. 
It remains to prove that $(T,\cV)$ displays the undominated ends in the closure of~$\nt$. 

By Lemma~\ref{lemma: tdc displays ends of G corrisponding to ends of T} it suffices to show that the undominated ends in the closure of $\nt$ are precisely the ends of $G$ that correspond  to the ends of $T$.
For the forward inclusion, consider any undominated end $\omega$ in the closure of~$\nt$. 
By (iii), it follows that $\omega$ lives at a node $t_n$ of $T_n$ (with regard to $(T_n,\alpha_n)$) at level $n$ for every $n\in\N$, and these nodes form a ray $R=t_0t_1\ldots$ of $T$. Then $\omega$ corresponds to the end of $T$ containing $R$.

For an indirect proof of the backward inclusion, consider any end $\omega$ of $G$ that is either dominated or not contained in the closure of $\nt$. 
We show that $\omega$ does not correspond to any end of $T$. 
If $\omega$ is dominated, then this follows from the fact that $(T,\cV)$ has upwards disjoint finite separators. 
Otherwise $\omega$ is not contained in the closure of $\nt$. Let $n\in \N$ be strictly larger than the height of~$\omega$.
By (iv), it follows that $\omega$ lives at a node $t_\omega$ of $T_n$ of height $<n$ with regard to $(T_n,\alpha_n)$. 
And by (ii), the tree $T_n$ consists precisely of the first $n$ levels of $T$. We conclude that $\omega$ lives in the part of $(T,\cV)$ corresponding to~$t_\omega$.

\begin{figure}[ht]
\begin{center}
\def\svgwidth{8cm}
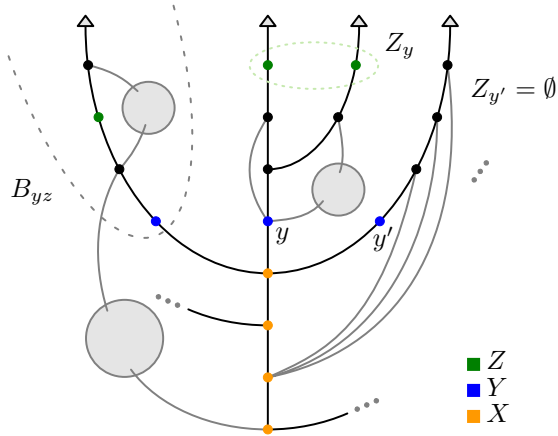
\end{center}
\caption{The construction of the $(T_n,\alpha_n)$ in the proof of Theorem~\ref{NormalTreeGivesCarmesinTDC}. Here the vertex set $Z$ consists of all vertices that are contained in some $Z_y$ with~$y\in Y$. The depicted tree is $\nt$.}
\label{fig:short_carmesin}
\end{figure}

Now, we turn to the construction of the $(T_n,\alpha_n)$, also see Figure~\ref{fig:short_carmesin}. 
At step $n+1$ suppose that 
$(T_n,\alpha_n)$ has already been defined and recall that the separators of $(T_n,\alpha_n)$ are vertex sets of ascending paths in $\nt$ by (i). 
Let $L$ be the $n$th level of $T_n$. 
To obtain $(T_{n+1},\alpha_{n+1})$ from $(T_{n} ,\alpha_n)$, we will add for each $\ell\in L$ new vertices (possibly none) to $T_n$ that we join exactly to $\ell$ and define the image of the so emerging edges under $\alpha_{n+1}$. 
So fix $\ell\in L$. 
Let $X$ be the separator of the separation corresponding to the edge between $\ell$ and its predecessor in $T_n$ (if $n=0$ put $X=\emptyset$). 
Recall that $X$ is the vertex set of an ascending path in $\nt$ by (i).
In $\nt$, let $Y$ be the set of up-neighbours of the maximal vertex in $X$ (for $n=0$ let $Y:=\{r\}$).
For each $y\in Y$ let $Z_y$ be the set of those $z\in \uc{y}_{\nt}$ that are minimal with the property that $G$ contains no $\nt$-path starting in $\dc{y}_{\nt}$ and ending in~$\uc{z}_{\nt}$. 
(Note that a normal ray of $\nt$ that contains $y$ meets $Z_y$ if and only if it is not dominated by any of the vertices in $\dc{y}_{\nt}$; this fact together with (i) will guarantee (iii) for $n+1$.) 
Then the vertex set of $y\nt z$ separates the connected sets $A_{yz}:= (V\setminus \guc{z}_{\nt})\cup V(y\nt z)$ and $B_{yz}:= V(y\nt z)\cup \guc{z}_{\nt }$
whenever $y\in Y$ and $z\in Z_y$. 
Join a node $t_{yz}$ to $\ell$ for every pair $(y,z)$ with $y\in Y$ and $z\in Z_y$, and put $\alpha_{n+1}(\ell t_{yz}):= (A_{yz},B_{yz})$. 
Then the $\Sinf$-tree $(T_{n+1},\alpha_{n+1})$ clearly satisfies (i) and~(ii).  That it satisfies (iii) was already argued in the construction and (iv) follows from (i) and the definition of $\alpha_{n+1}(\ell t_{yz})$. 
\end{proof}

With Theorem~\ref{NormalTreeGivesCarmesinTDC} at hand, we are finally able to prove Theorem~\ref{thm: domcomb tdc dual II}:

\begin{proof}[Proof of Theorem~\ref{thm: domcomb tdc dual II}] First, we show that (i) and (ii) cannot hold at the same time. 
For this, assume for a contradiction that $G$ contains a dominated comb \at $U$ and has a tree-decomposition $(T,\cV)$ with upwards disjoint finite separators that displays $\Abs{U}$.
We write $\omega$ for the end of $G$ containing the comb's spine.
Then $\omega$ lies in the closure of $U$, and since $(T,\cV)$ displays $\Abs{U}$ there is a unique end $\eta$ of $T$ to which $\omega$ corresponds.
But as the finite separators of $(T,\cV)$ are upwards disjoint, it follows that $\omega$ is undominated in~$G$, contradicting that $\omega$ contains the spine of a dominated comb.

Now, to show that at least one of (i) and (ii) holds, we prove $\neg$(i)$\to$(ii).
Using Theorem~\ref{domCombNTduality} we find a normal tree $\nt \subset G$ that contains $U$ cofinally and all whose rays are undominated in $G$.
Furthermore, by the `moreover' part of Theorem~\ref{domCombNTduality} we may assume that every component of $G-\nt $ has finite neighbourhood, and by Lemma~\ref{lemma: closure and tree containing U cofinally} we have $\Abs{U}=\Abs{\nt}$.
Then Theorem~\ref{NormalTreeGivesCarmesinTDC} yields a rooted tree-decomposition $(T',\cV')$ of $G$ as in (ii) that has connected separators and covers $V(\nt)$ cofinally.
It remains to show that $(T',\cV')$ can be chosen so as to cover $U$ cofinally.
For this, consider the nodes of $T'$ whose parts meet $U$, and let $T\subset T'$ be induced by their down-closure in $T'$.
Then let $(T',\alpha')$ be the $\Sinf$-tree of $G$ that corresponds to $(T',\cV')$ and consider the rooted tree-decomposition $(T,\cV)$ of $G$ that corresponds to $(T,\,\alpha'\rest\vE(T)\,)$.
Now $(T,\cV)$ is as in (ii) and satisfies the theorem's `moreover' part.
\end{proof}

\subsection{A duality theorem related to stars and combs}\label{subsection: domcomb combined}
Finally, we present a duality theorem for dominated combs in terms of tree-decompositions that makes both the left and the right dashed arrow in Figure~\ref{fig: desired rel duality thms in terms of tdcs} true. In order to state the theorem, we need one more definition. 
A rooted tree-decomposition $(T,\cV)$ of a graph $G$ has \emph{essentially disjoint separators} if there is an edge set $F\subset E(T)$ meeting every ray of $T$ infinitely often such that the separators of $(T,\cV)$ associated with the edges in $F$ are upwards disjoint.

\newpage
\begin{mainresult}\label{thm: domcomb tdc dual III}
\TFAD
\begin{enumerate}
    \item $G$ contains a dominated comb \at $U$;
    \item $G$ has a rooted tree-decomposition $(T,\cV)$ such that:
    \begin{itemize}
        \item[\textbf{--}]each part contains at most finitely many vertices from $U$;
        \item[\textbf{--}]all parts at non-leaves of $T$ are finite;
        \item[\textbf{--}] $(T,\cV)$ has essentially disjoint connected separators;
        \item[\textbf{--}] $(T,\cV)$ displays the ends in the closure of~$U$.
    \end{itemize}
\end{enumerate}
\end{mainresult} 

Before we provide a proof of this theorem, let us see that it relates to the duality theorems for stars and combs in terms of tree-decompositions as desired (also see Figure~\ref{fig: domcomb tdc dual III}, which shows Figure~\ref{fig: desired rel duality thms in terms of tdcs} in greater detail and where Theorem~\ref{thm: domcomb tdc dual III}~(ii) including the theorem's `moreover' part is inserted for `?').

\begin{figure}[ht]
\begin{tikzcd}[column sep=0cm] 
& {\begin{tabular}{l}
$\nexists$ dominated comb \\ \phantom{$\nexists$ }\at~$U$\end{tabular}}\arrow[d,leftrightarrow]\\ 
{\begin{tabular}{l}
$\nexists$ comb \at~$U$\end{tabular}}\arrow[d,leftrightarrow]\arrow[ur] & {\begin{tabular}{l}
    $\exists$ tree-decomposition\\ \phantom{$\nexists$ }with $(\ast)$,  essentially \\ \phantom{$\nexists$}   disjoint separators   \\\phantom{$\nexists$}  and parts at non-\\
    \phantom{$\nexists$} leaves finite
\end{tabular}}\arrow[dl,rightarrow,dashed]\arrow[dr,rightarrow,dashed] & {\begin{tabular}{l}
$\nexists$ star \at $U$\end{tabular}}\arrow[ul,rightarrow]\arrow[d,leftrightarrow]\\
{\begin{tabular}{l}
$\exists$ rayless tree-decom-\\ \phantom{$\exists$ }position with ($\ast$)\\
\phantom{$\exists$ }and parts at non-\\
\phantom{$\exists$ }leaves finite
\end{tabular}
} & & {\begin{tabular}{l}
$\exists$ locally finite tree-\\ \phantom{$\exists$} decomposition with\\
\phantom{$\exists$} ($\ast$) and pairwise\\ \phantom{$\exists$} disjoint separators\end{tabular}}
\end{tikzcd}
    \caption{The relation between the duality theorems for combs, stars and the final duality theorem for the dominated combs in terms of tree-decompositions.\newline
    Condition ($\ast$) says that parts contain at most finitely many vertices from $U$, that the separators are finite and connected, and that the tree-decomposition displays $\Abs{U}$.}
    \label{fig: domcomb tdc dual III}
\end{figure}
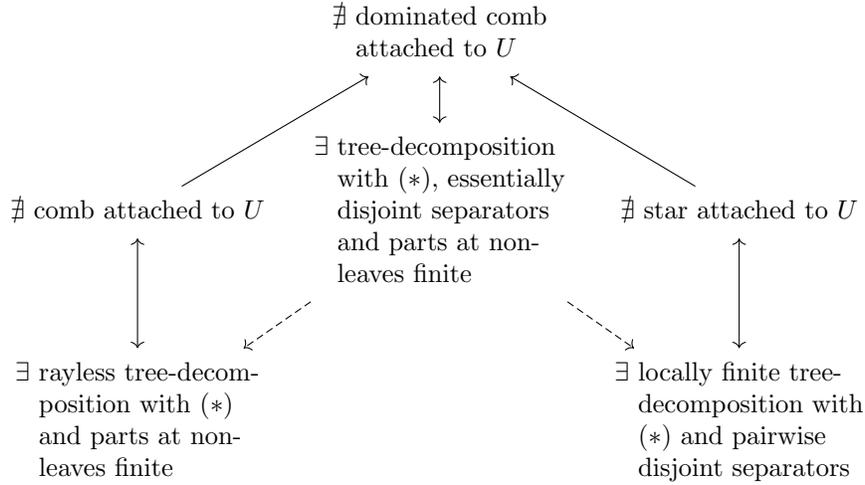

On the one hand, if $G$ does not contain a comb \at $U$, then in particular it does not contain a dominated comb \at $U$. Hence Theorem~\ref{thm: domcomb tdc dual III} returns a tree-decomposition $(T,\cV)$. By our assumption that there is no comb \at $U$, and since $(T,\cV)$ displays $\Abs{U}$, it follows that the decomposition-tree $T$ is rayless. We conclude that $(T,\cV)$ is as in Theorem~\ref{thm: comb tree-decomposition duality}~(ii) including the theorem's `moreover' part.

On the other hand, if $G$ does not contain a star \at $U$, then in particular it does not contain a dominated comb \at $U$. 
Hence Theorem~\ref{thm: domcomb tdc dual III} returns a rooted tree-decomposition $(T,\cV)$ that, in particular, has essentially disjoint finite connected separators and displays~$\Abs{U}$.
Write $(T,\alpha)$ for the $\Sinf$-tree that corresponds to $(T,\cV)$.
Let $F\subseteq E(T)$ witness that $(T,\cV)$ has essentially disjoint separators.
By possibly thinning out $F$, we may assume that each edge in $F$ meets a rooted ray of $T$.
Consider the tree $\tilde{T}$ that is obtained from $T$ by contracting all the edges of $T$ that are not in $F$ and let $\tilde{\alpha}$ be the restriction of $\alpha$ to $\vF=\vE(\tilde{T})$.
Then $(\tilde{T},\tilde{\alpha})$ corresponds to a tree-decomposition $(\tilde{T},\cW)$ of $G$ with upwards disjoint finite connected separators that displays $\Abs{U}$.
Thus, the tree-decomposition $(\tilde{T},\cW)$ is one of the tree-decompositions of $G$ that are complementary to dominated combs as in Theorem~\ref{thm: domcomb tdc dual II}~(ii) including the theorem's `moreover' part (it covers $U$ cofinally as $F$ meets every rooted ray of $T$ while $(T,\cV)$ displays $\Abs{U}$).
Then, as we have already argued below Theorem~\ref{thm: domcomb tdc dual II}, the tree-decomposition $(\tilde{T},\cW)$ can easily be turned into a tree-decomposition 
as in Theorem~\ref{thm: star tree-decomposition duality}~(ii) including the theorem's `moreover' part.

\begin{proof}[Proof of Theorem~\ref{thm: domcomb tdc dual III}]
Since the tree-decomposition  from (ii) displays $\Abs{U}$ and has essentially disjoint finite separators, it follows by standard arguments that not both (i) and (ii) can hold at the same time. 

In order to show that at least one of (i) and (ii) holds, we prove $\neg$(i)$\to$(ii).
For this, suppose that $G$ contains no dominated comb \at $U$.
Using Theorem~\ref{domCombNTduality} we find a normal tree $\nt\subset G$ that contains $U$ cofinally and such that every component of $G-\nt$ has finite neighbourhood.
Then we let $(T',\cW)$ be a rooted tree-decomposition obtained by and constructed like in the proof of Theorem~\ref{NormalTreeGivesCarmesinTDC}.
Now $(T',\cW)$ is almost as desired, and even has upwards disjoint separators; however, its parts may be infinite.
To solve this, we define another tree-decomposition $(T,\cV)$, as follows.

If $C$ is a component of $G-\nt$, then we write $t_C$ for the maximal node $t\in\nt$ in its neighbourhood.
We let $T$ be the rooted tree obtained from $\nt$ by adding each component $C$ of $G-\nt$ as a new node that we join precisely to~$t_C$.
For every node $t\in\nt\subset T$ we let $x(t)$ be the least node of $T'$ with $t\in W_{x(t)}$ and define $V_t:=\dc{t}_{\nt}\cap W_{x(t)}$.
For every component $C\in T-\nt$ we define $V_C:=V(C)\cup V_{t_C}$.
To see that $(T,\cV)$ is a rooted tree-decomposition of~$G$, note that for every vertex $v$ of $G$ the set $\{\,t\in T\mid v\in V_t\,\}$ induces a subtree of~$T$.

We claim that $(T,\cV)$ is as desired.
The only property in~(ii) that is not evident is that $(T,\cV)$ has essentially disjoint separators.
For every edge $e\in T'$ we write $P_e$ for the ascending path in $\nt$ ending in $t_e$ whose vertex set is the separator associated with $e$ in the tree-decomposition $(T',\cW)$.
To see that the separators of $(T,\cV)$ are essentially disjoint, consider for every edge $e\in T'$ the set $F_e$ of up-edges $f\in\nt$ at $t_e$ with separators of $(T',\cW)$ above them.
Then every edge $f\in F_e\subset E(T)$ is associated with the separator $V(P_e)$ in~$(T,\cV)$ that is also the separator associated with $e$ in~$(T',\cW)$.
So the separators of~$(T,\cV)$ associated with the edges in $F:=\bigcup\,\{\,F_e\mid e\in T'\,\}$ are upwards disjoint because the separators of~$(T',\cW)$ are. 
Since $\nt$ contains $U$ cofinally and $G$ contains no dominated comb \at ~$U$, every normal ray of~$\nt$ is undominated and therefore passes through infinitely many edges of~$F$.
Hence every ray of~$T$ passes through infinitely many edges of~$F$ as well.
\end{proof}

\begin{example}
The tree-decomposition in Theorem~\ref{thm: domcomb tdc dual III}~(ii) cannot be chosen with pairwise disjoint separators instead of essentially disjoint separators: 
Suppose that $G$ consists of the first three levels of $T_{\aleph_0}$ and let $U:=V(G)$. 
Then $G$ contains no comb \at $U$. 
In particular, as we have already argued in the text below Theorem~\ref{thm: domcomb tdc dual III}, every tree-decomposition $(T,\cV)$ of $G$ complementary to dominated combs as in Theorem~\ref{thm: domcomb tdc dual III} is also a tree-decomposition of $G$ complementary to combs as in Theorem~\ref{thm: comb tree-decomposition duality}.
But then $(T,\cV)$ cannot be chosen with pairwise disjoint separators, as pointed out in~\mbox{\cite[Example~3.7]{StarComb1StarsAndCombs}}.
\end{example}

\bibliographystyle{amsplain}
\bibliography{StarCombBib}

\providecommand{\bysame}{\leavevmode\hbox to3em{\hrulefill}\thinspace}
\providecommand{\MR}{\relax\ifhmode\unskip\space\fi MR }
\providecommand{\MRhref}[2]{%
  \href{http://www.ams.org/mathscinet-getitem?mr=#1}{#2}
}
\providecommand{\href}[2]{#2}
\begin{thebibliography}{1}

\bibitem{StarComb1StarsAndCombs}
C.~Bürger and J.~Kurkofka, \emph{{Duality theorems for stars and combs I:
  Arbitrary stars and combs}}, 2020,
  \href{https://arxiv.org/abs/2004.00594}{arXiv:2004.00594}.

\bibitem{StarComb3TheUndominatedComb}
\bysame, \emph{{Duality theorems for stars and combs III: Undominated combs}},
  2020, \href{https://arxiv.org/abs/2004.00592}{arXiv:2004.00592}.

\bibitem{StarComb4TheUndominatingStar}
\bysame, \emph{{Duality theorems for stars and combs IV: Undominating stars}},
  2020, \href{https://arxiv.org/abs/2004.00591}{arXiv:2004.00591}.

\bibitem{carmesin2014all}
J.~Carmesin, \emph{All graphs have tree-decompositions displaying their
  topological ends}, Combinatorica \textbf{39} (2019), no.~3, 545--596.

\bibitem{DiestelBook5}
R.~Diestel, \emph{{Graph Theory}}, 5th ed., Springer, 2016.

\bibitem{VTopComp}
\bysame, \emph{End spaces and spanning trees}, J. Combin.\ Theory (Series B)
  \textbf{96}
  (\href{http://www.math.uni-hamburg.de/home/diestel/papers/EndSpacesAndNSTs.pdf}{2006}),
  no.~6, 846--854.

\bibitem{halin78}
R.~Halin, \emph{Simplicial decompositions of infinite graphs}, Advances in
  Graph Theory, Annals of Discrete Mathematics (B.Bollob\'as, ed.), vol.~3,
  North-Holland, 1978.

\bibitem{jung69}
H.A. Jung, \emph{{Wurzelb{\"a}ume und unendliche Wege in Graphen}}, Math.\
  Nachr. \textbf{41} (1969), 1--22.

\end{thebibliography}
\end{document}